\documentclass[reqno,12pt]{amsart}
\usepackage{amsmath,amsthm,amsfonts,amssymb}
\usepackage[utf8x]{inputenc}
\usepackage[english]{babel}
\usepackage[normalem]{ulem} 
\usepackage{graphicx}

\date{} 
\newtheorem{theorem}{Theorem}[section]

\newtheorem{proposition}[theorem]{Proposition}
\newtheorem{remark}[theorem]{Remark}
\newtheorem{corollary}[theorem]{Corollary}
\newtheorem{example}[theorem]{Example}
\numberwithin{equation}{section}

\renewcommand{\mod}{\,\mathrm{mod}\,}

\DeclareMathOperator{\pr}{pr}

\begin{document}

\title[]{On uniqueness of an optimal solution to the Kantorovich problem with density constraints}
\maketitle

\begin{center}
S.N. Popova~\footnote{Moscow Institute of Physics and Technology; National Research University Higher School of Economics.}
\end{center}

\vskip .2in

{\bf Abstract.}
We study optimal transportation problems with constraints on densities of transport plans. 
We obtain a sharp condition for the uniqueness of an optimal solution to the Kantorovich problem with density constraints, 
namely that the Borel measurable cost function $h(x, y)$ satisfies the following non-degeneracy condition:  
$h(x, y)$ can not be expressed as a sum of functions $u(x) + v(y)$ on a set of positive measure.


\section{Introduction}

Over the last decades the Monge-Kantorovich optimal transportation problem has attracted considerable attention of researchers and
new interesting modifications of this problem have been investigated. 
Given two Borel probability measures $\mu$ and $\nu$ on topological spaces $X$ and $Y$ respectively 
and a nonnegative Borel function $h$ on~$X\times Y$, the Kantorovich optimal transportation problem concerns the minimization of  the integral
$$
I_h(\sigma) = \int_{X \times Y} h(x, y) \, \sigma(dx dy)
$$
over all measures $\sigma$ in the set $\Pi(\mu,\nu)$ consisting of Borel probability measures on $X\times Y$
with projections $\mu$ and $\nu$ on the factors, that is, $\sigma (A\times Y)=\mu(A)$ and $\sigma (X\times B)=\nu(B)$
for all Borel sets $A\subset X$ and $B\subset Y$. The measures $\mu$ and $\nu$ are called marginal distributions
or marginals, and $h$ is called a cost function.
If the cost function $h$ is continuous (or at least lower semicontinuous)  
and bounded and the measures $\mu$, $\nu$ are Radon, then the functional $I_h$ attains a minimum on $\Pi(\mu, \nu)$. The measures on which the minimum is attained are called
optimal measures or optimal Kantorovich plans. If $X = Y = \mathbb R^n$, $h(x, y) = g(|x - y|)$, where $g \colon \mathbb R \to \mathbb R$ is a strictly convex function,
and $\mu$ is absolutely continuous with respect to the Lebesgue measure, then the optimal Kantorovich plan is unique and is generated by a Monge mapping (see \cite{GM}). 
General information about Monge and Kantorovich problems can be found in \cite{AG}, \cite{B22}, \cite{BKP}, \cite{RR}  and~\cite{V09}.

In \cite{KM1}--\cite{KM4} a natural modification of the Kantorovich problem was studied (yet in the situation of relatively simple spaces
such as manifolds), which deals with the minimization of the functional $I_h$ 
under the additional constraint that the admissible plans belong to the subset
$\Pi_\Phi(\mu,\nu)$ of  $\Pi(\mu,\nu)$
consisting of measures $\sigma$ possessing densities with respect to a fixed measure $\lambda$
on $X\times Y$ (Lebesgue measure in the first papers) such that these densities are dominated
 by a given function~$\Phi$:
 $$
 \frac{d\sigma}{d\lambda}\le \Phi.
 $$
In \cite{KM1}--\cite{KM2} the results on existence and uniqueness of an optimal solution were obtained for absolutely continuous measures with respect to the Lebesgue measure on $\mathbb R^n$ 
and bounded cost functions $h$ which belong to the class $C^2$ everywhere on $\mathbb R^n \times \mathbb R^n$ except for a closed subset of Lebesgue measure $0$.  
In the works \cite{KM3}--\cite{KM4} the duality theory for the Kantorovich problem with density constraints was studied. 
These investigations have been continued in \cite{D18}, \cite{BDM21}, where general measurable spaces have been considered
and $\lambda$ has been an arbitrary probability measure on $X\times Y$. In this situation, the existence of a minimum has been proved
under the assumption that $\Phi\in L^1(\lambda)$ and $\Pi_\Phi(\mu,\nu)$ is not empty. In particular, no topological conditions
on $\Phi$ are involved. In \cite{BPR} Kantorovich problems with density constraints on transport plans and nonlinear cost functionals have been studied. The papers \cite{BDM21}, \cite{B22}, \cite{BP} addressed Kantorovich problems with density constraints depending on a parameter. 

In this work we study Kantorovich problem with density constraints for general Borel probability measures on Souslin spaces and Borel measurable cost functions (without any regularity assumptions). 
We establish the uniqueness of an optimal solution to the Kantorovich problem with density constraints
for Borel measurable cost functions $h(x, y)$ which satisfy the following non-degeneracy condition: 
$h(x, y)$ can not be expressed as a sum of functions $u(x) + v(y)$ on a set of positive measure.
The non-degeneracy of the function $h$ yields a sharp condition for the uniqueness of an optimal solution: if this condition is not satisfied, 
then we can construct measures $\mu$ and $\nu$ for which the optimal solution to the Kantorovich problem with density constraints is not unique. 
In Section 2 we prove our main theorem about the uniqueness of an optimal solution to the Kantorovich problem with density constraints. 
Section 3 contains auxiliary statements related to the condition of non-degeneracy of a cost function. In Section 3 we also provide a wide class of cost functions satisfying the condition of non-degeneracy.

\section{Uniqueness of an optimal solution to the Kantorovich problem with density constraints}

Let $X, Y$ be topological spaces and let a measure $\eta \in \mathcal P(X \times Y)$ be fixed. For any function $\Phi \in L^1(\eta)$ denote by 
$\Pi_{\Phi}(\mu, \nu; \eta)$ the subset of $\Pi(\mu, \nu)$ obtained by imposing the following density constraints:  
$$
\Pi_{\Phi}(\mu, \nu; \eta) = \Bigl\{ \sigma \in \Pi(\mu, \nu): \frac{d\sigma}{d\eta} \le \Phi \Bigr\}. 
$$

Consider the Kantorovich problem with density constraints:
\begin{equation}\label{prob}
\int_{X \times Y} h(x, y) \sigma(dx dy) \to \inf, \quad \sigma \in \Pi_{\Phi}(\mu, \nu; \eta). 
\end{equation}

Denote
$$
I_h(\sigma) = \int_{X \times Y} h(x, y) \sigma(dx dy).
$$

In \cite{BDM21} the following result on the existence of an optimal solution to the Kantorovich problem with density constraints was established. 

\begin{theorem}[\cite{BDM21}]{\rm (Existence)} \label{th_exist}
Let $(X, \mathcal A, \mu)$ and $(Y, \mathcal B, \nu)$ be probability spaces. Let $\eta \in \mathcal P(X \times Y)$ and $\Phi \in L^1(\eta)$. Suppose that $\Pi_{\Phi}(\mu, \nu; \eta)  \neq \varnothing$. 
Let $h \colon X \times Y \to \mathbb R$ be a bounded from below $\mathcal A \otimes \mathcal B$-measurable function. 
Then the functional $I_h$ attains a minimum on the set $\Pi_{\Phi}(\mu, \nu; \eta)$. 
\end{theorem}

The set of densities of measures from the set $\Pi_{\Phi}(\mu, \nu; \eta)$ with respect to the measure $\eta$ is compact in the weak topology of $L^1(\eta)$, and the cost functional $I_h$ is lower semicontinuous in the weak topology of $L^1(\eta)$, 
therefore, it attains a minimium on $\Pi_{\Phi}(\mu, \nu; \eta)$.

In studying of uniqueness of an optimal solution we employ the following condition of non-degeneracy of the cost function  $h$ with respect to the measure $\eta$. 

Let $X, Y$ be Souslin spaces (i.e., images of Polish spaces under continuous mappings) and let $\eta \in \mathcal P(X \times Y)$. 
We say that a Borel measurable function $h \colon X \times Y \to \mathbb R$ 
\textit{satisfies the non-degeneracy condition with respect to the measure $\eta$} if there do not exist a Borel set $A \subset X \times Y$ with $\eta(A ) > 0$ 
and functions  $u \colon X \to \mathbb R$ and $v \colon Y \to \mathbb R$ such that
$$h(x, y) = u(x) + v(y) \quad \forall (x, y) \in A. $$

For the proof of our main result on uniqueness of an optimal solution to the Kantorovich problem with density constraints we use some auxiliary statements concerning the condition of non-degeneracy of the cost function $h$ with respect to the measure $\eta$. 
The proofs of these statements are presented in Section 3. Moreover, Section 3 contains some examples and sufficient conditions for the non-degeneracy of a cost function.  

We prove the uniqueness of an optimal solution to the Kantorovich problem with density constraints
for Borel measurable cost functions satisfying the non-degeneracy condition. 

\begin{theorem}{\rm (Uniqueness)} \label{th_unique}
Let $X, Y$ be Souslin spaces, let $\mu \in \mathcal P(X)$, $\nu \in \mathcal P(Y)$. 
Let $\eta = \eta_1 \otimes \eta_2$, where $\eta_1 \in \mathcal P(X)$, $\eta_2 \in \mathcal P(Y)$ are non-atomic measures. 
Suppose that $\Phi \in L^1(\eta)$ and $\Pi_{\Phi}(\mu, \nu; \eta) \neq \varnothing$. 
Let $h \colon X \times Y \to \mathbb R$ be a bounded from below Borel measurable function which satisfies the following condition of non-degeneracy: 
there do not exist universally measurable functions $u \colon X \to \mathbb R$ and $v \colon Y \to \mathbb R$ such that
$$\eta(\{(x, y) \in X \times Y: h(x, y) = u(x) + v(y)\}) > 0. $$
Then an optimal solution to the problem (\ref{prob}) is unique and has the form 
$$\sigma~=~I_{W} \Phi \cdot \eta$$ for some Borel set $W \subset X \times Y$. 
\end{theorem}

\begin{proof}
Let $\sigma = \rho \cdot \eta$ be an optimal solution to the problem (\ref{prob}), where $\rho = \frac{d \sigma}{d \eta}$ is the density of the measure $\sigma$ with respect to $\eta$. 
We show that  
$$\eta(\{(x, y) \in X \times Y: 0 < \rho(x, y) < \Phi(x, y)\}) = 0.$$ 
Suppose the contrary. Then for some  
$\varepsilon > 0$ we have $\eta(\{(x, y): \varepsilon < \rho < \Phi - \varepsilon\}) > 0$. 
Denote
$$U = \{(x, y) \in X \times Y: \varepsilon < \rho(x, y) < \Phi(x, y) - \varepsilon\}.$$

It is known (see \cite{B07}) that every Souslin space with a non-atomic probability measure is isomorphic mod 0 to the space $([0, 1], \lambda)$, where $\lambda$ is the Lebesgue measure on $[0, 1]$. Then there exist sets  
$\tilde X \subset X$, $\tilde Y \subset Y$ with $\eta_1(X \setminus \tilde X) = \eta_2(Y \setminus \tilde Y) = 0$ and
point isomorphisms 
$f \colon \tilde X \to S$, $g \colon \tilde Y \to T$, where $\lambda([0, 1] \setminus S) = \lambda([0, 1] \setminus T) = 0$ and
$$\eta_1|_{\tilde X} \circ f^{-1} = \lambda|_{S}, \quad \eta_2|_{\tilde Y} \circ g^{-1} = \lambda|_{T}. $$

Define 
$$\hat h(s, t) = h(f^{-1}(s), g^{-1}(t)) \quad \forall s \in S, \, t \in T. $$ 
Let $$V = (f \otimes g)(U \cap (\tilde X \times \tilde Y)).$$
Note that $(\lambda \otimes \lambda)(V) = (\eta_1 \otimes \eta_2)(U \cap (\tilde X \times \tilde Y)) = (\eta_1 \otimes \eta_2)(U) > 0$. 
Denote $\lambda_2 = \lambda \otimes \lambda$. 
A.e. $z \in V$ is a density point of $V$, that is, 
$$\lim_{r \to 0} \frac{\lambda_2(V \cap B_r(z))}{\lambda_2(B_r(z))} = 1, $$
where $B_r(z)$ is a ball with the center $z \in \mathbb R^2$ and radius $r > 0$. 
Moreover, a.e. $z \in V$ is a Lebesgue point for the function $\hat h$, that is, 
$$
\lim_{r \to 0} \frac{\int_{B_r(z)} |\hat h(w) -\hat h(z)| \lambda_2(dw)}{\lambda_2(B_r(z))} = 0. 
$$

Let $\tilde V \subset V$ be the set of all points $z \in V$, which are density points for $V$ and are Lebesgue points for the function  $\hat h$.  
Denote $\tilde U = (f \otimes g)^{-1}(\tilde V)$. Then $\eta(\tilde U)  = \lambda_2(\tilde V) > 0$. Since the function $h$ satisfies the condition of non-degeneracy, by Propositions \ref{prop_nondegenerate_equiv1} and \ref{prop_nondegenerate_equiv2}  
there exist points $x_1, \dots, x_n \in \tilde X$, $y_1, \dots, y_n \in \tilde Y$ such that $(x_i, y_i), (x_{i+1}, y_i) \in \tilde U$ for all $i \in \{1, \dots, n\}$ (where $x_{n+1} := x_1$) and
$$
\sum_{i = 1}^{n} h(x_i, y_i) \neq \sum_{i = 1}^{n} h(x_{i+1}, y_i).
$$
Therefore, there exist points $s_1, \dots, s_n \in S$, $t_1, \dots, t_n \in T$, such that
$(s_i, t_i), (s_{i+1}, t_i) \in \tilde V$ for all $i \in \{1, \dots, n\}$ (where $s_{n+1} := s_1$) and
$$
\sum_{i = 1}^{n} \hat h(s_i, t_i) \neq \sum_{i = 1}^{n} \hat h(s_{i+1}, t_i).
$$
Without limitation of generality we may assume that
$$
\sum_{i = 1}^{n} \hat h(s_i, t_i) < \sum_{i = 1}^{n} \hat h(s_{i+1}, t_i).
$$

Let 
$$\Delta_r = \{z \in B_r(0, 0): (s_i, t_i) + z \in V, (s_{i+1}, t_i) + z \in V \mbox{  } \forall i \in \{1, \dots, n\}\}$$
be the set of all points $z \in B_r(0, 0)$ such that the shifts $(s_i, t_i) + z$, $(s_{i + 1}, t_i) + z$ belong to $V$ for all $i \in \{1, \dots, n\}$. 
Note that  
\begin{multline*}
\lambda_2 (\Delta_r) \ge \lambda_2(B_r(0, 0)) - \sum_{i = 1}^n \lambda_2(\{z \in B_r(0, 0): (s_i, t_i) + z \notin V\}) - \\ - \sum_{i = 1}^n \lambda_2(\{z \in B_r(0, 0): (s_{i + 1}, t_i) + z \notin V\}) = \\ = 
 \lambda_2(B_r(0, 0)) - \sum_{i = 1}^n \lambda_2(B_r(s_i, t_i) \setminus V) - \sum_{i = 1}^n \lambda_2(B_r(s_{i + 1}, t_i) \setminus V). 
\end{multline*}
Since $(s_i, t_i)$ and $(s_{i+1}, t_i)$ are density points for the set $V$, 
$$\frac{\lambda_2(B_r(s_i, t_i) \setminus V)}{\lambda_2(B_r(s_i, t_i))} \to 0, \quad r \to 0.$$
Therefore, 
$$
\frac{\lambda_2 (\Delta_r)}{\lambda_2(B_r(0, 0))} \to 1, \quad r \to 0. 
$$
Thus, there exists $\delta > 0$ such that for all $r < \delta$ we have $\lambda_2(\Delta_r) > 0$. Moreover, we may assume that $r$ is sufficiently small so that all the balls $B_r(s_i, t_i)$, $B_r(s_{i + 1}, t_i)$, $i = 1, \dots, n$, are disjoint. 

We show that we can build a small perturbation $\tilde \rho$ of the function $\rho$ on the set $U$ such that
$\tilde \rho \cdot \eta \in \Pi_{\Phi}(\mu, \nu; \eta)$ and $I_h(\tilde \rho \cdot \eta) < I_h(\rho \cdot \eta)$. 
For this purpose we construct a function $\theta_r$ which is used to determine the perturbation of the function $\rho$. 

Define a function $\hat \theta_r \colon S \times T \to \mathbb R$ by the following formula: 
$$\hat \theta_r(s, t) = 1 \quad \mbox{if} \quad (s, t) \in (s_i, t_i) + \Delta_r, $$ 
$$\hat \theta_r(s, t) = -1 \quad \mbox{if} \quad (s, t) \in (s_{i+1}, t_i) + \Delta_r,$$ 
and $\hat \theta_r(s, t) = 0$ otherwise. 

From the definition of the set $\Delta_r$ it follows that $\{(s, t) \in S \times T: \hat \theta_r(s, t) \neq 0\} \subset V$.

Note that
$$
\int_T \hat \theta_r(s, t) \lambda(dt) = \int_S \hat \theta_r(s, t) \lambda(ds) = 0
$$
for all $s \in S$ and $t \in T$. Indeed, for any $t \in T$ we have
\begin{multline*}
\int_S \hat \theta_r(s, t) \lambda(ds) = \\ = \sum_{i = 1}^n \bigl(\lambda(\{s \in S: (s, t) \in (s_i, t_i) + \Delta_r\}) - \lambda(\{s \in S: (s, t) \in (s_{i + 1}, t_i) + \Delta_r\})\bigr) = 0, 
\end{multline*}
since
$$
\lambda(\{s \in S: (s, t) \in (s_i, t_i) + \Delta_r\}) = \lambda \bigl(\Delta_r^{t - t_i} \bigr), 
$$
where $\Delta_r^{t - t_i} = \{s \in S: (s, t - t_i) \in \Delta_r\}$.

Define 
$$\theta_r(x, y) = \hat \theta_r(f(x), g(y))$$ 
and 
$$\rho_r(x, y) = \rho(x, y) + \varepsilon \theta_r(x, y). $$
Show that $\rho_r \cdot \eta \in \Pi_{\Phi}(\mu, \nu; \eta)$. 
We have $|\theta_r| \le 1$ and 
$$\{(x, y) \in X \times Y: \theta_r(x, y) \neq 0\} \subset U.$$
Therefore, $\rho_r = \rho$ on $X \setminus U$ and $\rho - \varepsilon \le \rho_r \le \rho + \varepsilon$ on $U$. Thus, $0 \le \rho_r \le \Phi$. 
Furthermore, the projections of the measure 
$\rho_r \cdot \eta$ on $X$ and $Y$ are equal to $\mu$ and $\nu$ respectively. This follows from the fact that for $\eta_1$-a.e. $x \in X$ we have
$$
\int_Y \theta_r(x, y) \eta_2(dy) = \int_T \hat \theta_r(f(x), t) \lambda(dt) = 0
$$
and for $\eta_2$-a.e. $y \in Y$ we have
$$
\int_X \theta_r(x, y) \eta_1(dx) = \int_S \hat \theta_r(s, g(y)) \lambda(ds) = 0. 
$$
Therefore, $\rho_r \cdot \eta \in \Pi_{\Phi}(\mu, \nu; \eta)$.

Let us show that $I_h(\rho_r \cdot \eta) < I_h(\rho \cdot \eta)$ for sufficiently small $r$. 
We have
\begin{multline*}
I_h(\rho_r \cdot \eta) - I_h(\rho \cdot \eta) = 
\varepsilon \int_{X \times Y} h(x, y) \theta_r(x, y) \eta(dx dy) = \\ = 
\varepsilon \int_{S \times T} \hat h(s, t) \hat \theta_r(s, t) \lambda_2(ds dt) =  \\ = 
\varepsilon \Biggl(
\sum_{i = 1}^n \int_{(s_i, t_i) + \Delta_r} \hat h(s, t) \lambda_2(ds dt) - 
\sum_{i = 1}^n \int_{(s_{i + 1}, t_i) + \Delta_r} \hat h(s, t) \lambda_2(ds dt) \Biggr). 
\end{multline*}
Note that 
\begin{multline*}
\Biggl|\frac{\int_{(s_i, t_i) + \Delta_r} \hat h(s, t) \lambda_2(ds dt)}{\lambda_2(B_r(s_i, t_i))} - 
\frac{\int_{(s_i, t_i) + \Delta_r} \hat h(s_i, t_i) \lambda_2(ds dt)}{\lambda_2(B_r(s_i, t_i))}\Biggr| \le \\
\le \frac{\int_{B_r(s_i, t_i)}|\hat h(s, t) - \hat h(s_i, t_i)| \lambda_2(ds dt)}{\lambda_2(B_r(s_i, t_i))} \to 0, \quad r \to 0, 
\end{multline*}
since $(s_i, t_i) \in \tilde V$ is a Lebesgue point for the function $\hat h$. 
Moreover, 
$$
\frac{\int_{(s_i, t_i) + \Delta_r} \hat h(s_i, t_i) \lambda_2(ds dt)}{\lambda_2(B_r(s_i, t_i))} = 
\hat h(s_i, t_i) \frac{\lambda_2(\Delta_r)}{\lambda_2(B_r(0, 0))} \to \hat h(s_i, t_i), \quad r \to 0.
$$
Therefore, 
\begin{multline*}
\frac{1}{\lambda_2(B_r(0, 0))} \sum_{i = 1}^n \Bigl(\int_{(s_i, t_i) + \Delta_r} \hat h(s, t) \lambda_2(ds dt) - 
\int_{(s_{i + 1}, t_i) + \Delta_r} \hat h(s, t) \lambda_2(ds dt)\Bigr) \to \\ \to
\sum_{i = 1}^n (\hat h(s_i, t_i) - \hat h(s_{i+1}, t_i)), \quad r \to 0.
\end{multline*}

Thus  
$$
I_h(\rho_r \cdot \eta) - I_h(\rho \cdot \eta) = \varepsilon \lambda_2(B_r(0, 0)) 
\Bigl(\sum_{i = 1}^n (\hat h(s_i, t_i) - \hat h(s_{i+1}, t_i)) + o(1) \Bigr) < 0
$$
for sufficiently small $r$. 
We arrive to a contradiction with the optimality of $\rho$. Therefore, $\eta(0 < \rho < \Phi) = 0$ and 
$\rho = I_W \Phi$ for some Borel set $W \subset X \times Y$. 

This implies the uniqueness of an optimal solution. Indeed, suppose that $\sigma_1$ and $\sigma_2$ are optimal solutions to the problem (\ref{prob}). 
Then $(\sigma_1 + \sigma_2)/2 \in \Pi_{\Phi}(\mu, \nu; \eta)$ is also an optimal solution to the problem (\ref{prob}). 
As shown above, $(\sigma_1 + \sigma_2)/2 = I_W \Phi \eta$ for some Borel set $W \subset X \times Y$. Therefore, $\sigma_1 = \sigma_2$. 
\end{proof}

Note that the condition of the non-degeneracy of the cost function also appears to be necessary for the uniqueness of an optimal solution. 
Indeed, suppose that $h(x, y) = u(x) + v(y)$ on some set $A \subset X \times Y$ with $\eta(A) > 0$. 
Then for any measure $\sigma \ll \eta$ with $\sigma(A) = 1$ there exist another measure  
$\pi \ll \eta$ such that $\pi(A) = 1$ and $\pi$ has the same projections on $X$ and $Y$ as $\sigma$
(denote these projections by $\tilde \mu$ and $\tilde \nu$ respectively). Note that
$$
\int_{X \times Y} h d\sigma = \int_{X \times Y} (u(x) + v(y)) d \sigma = 
\int_X u d\tilde \mu + \int_Y v d \tilde \nu = \int_{X \times Y} h d\pi. 
$$
Then for the function $\Phi = (\frac{d \sigma}{d \eta} + \frac{d \pi}{d \eta}) I_A$ we obtain that all measures from the set 
$\Pi_{\Phi}(\tilde \mu, \tilde \nu; \eta)$ are optimal for the functional $I_h$ and $|\Pi_{\Phi}(\tilde \mu, \tilde \nu; \eta)| > 1$, that is, an optimal plan is not unique.

It is a question for future research how to characterize the set $W \subset X \times Y$ such that the unique optimal solution has the form $\sigma = I_W \Phi \eta$. 
It is interesting to study the geometric and topological properties of $W$ and its boundary $\partial W$.

\section{Non-degeneracy of a cost function with respect to a measure}

We state the following necessary and sufficient condition for the non-degeneracy of a Borel measurable cost function $h \colon X \times Y \to \mathbb R$ with respect to a measure $\eta \in \mathcal P(X \times Y)$.

\begin{proposition}\label{prop_nondegenerate_equiv1}
Let $X, Y$ be Souslin spaces, let $\eta \in \mathcal P(X \times Y)$. The following conditions are equivalent 
for a Borel measurable function $h \colon X \times Y \to \mathbb R$: 

1) There exist a Borel set $A \subset X \times Y$ with $\eta(A) > 0$ and functions $u \colon X \to \mathbb R$ and $v \colon Y \to \mathbb R$ such that

$$h(x, y) = u(x) + v(y) \quad \forall (x, y) \in A. $$

2) There exists a Borel set $A \subset X \times Y$ with $\eta(A) > 0$ such that
for any $n \in \mathbb N$ and for any points $x_1, \dots, x_n \in X$, $y_1, \dots, y_n \in Y$ which satisfy the condition that
$(x_i, y_i), (x_{i+1}, y_i) \in A$ for all $i \in \{1, \dots, n\}$ (where $x_{n+1} := x_1$), the following equality holds: 
$$
\sum_{i = 1}^n h(x_i, y_i) = \sum_{i = 1}^n h(x_{i + 1}, y_i).
$$

\end{proposition}

\begin{proof}
Obviously, 1) $\Rightarrow$ 2). Let us prove that 2) $\Rightarrow$ 1). Let $A \subset X \times Y$ be a Borel set of positive 
measure for which the condition 2) is satisfied. Let us construct functions $u \colon \pr_X A \to \mathbb R$ and $v \colon \pr_Y A \to \mathbb R$
such that $h(x, y) = u(x) + v(y)$ on $A$, where $\pr_X A$ and $\pr_Y A$ are the projections of the set $A$ on $X$ and $Y$ respectively, i.e., 
$\pr_X A = \{x \in X: \exists y \in Y, \, (x, y) \in A\}$. Without limitation of generality we may assume that $\pr_X A = X$ and $\pr_Y A = Y$.

Consider the following equivalence relation on $X$:
$x \sim_X x'$ if there exist points $x_1 = x, x_2, \dots, x_{n+1} = x' \in X$ and $y_1, \dots, y_n \in Y$ such that 
$(x_i, y_i), (x_{i+1}, y_i) \in A$ for any $i \in \{1, \dots, n\}$.  
For any $x, x' \in X$ with $x \sim_X x'$, we set
\begin{equation} \label{Delta_u}
\Delta_u(x, x') = \sum_{i = 1}^{n} (h(x_i, y_i) - h(x_{i+1}, y_i)), 
\end{equation}
where $x_1 = x$, $x_{n + 1} = x'$ and $x_2, \dots, x_n \in X$, $y_1, \dots, y_n \in Y$ are such that 
$(x_i, y_i), (x_{i + 1}, y_i) \in A$ for all $i \in \{1, \dots, n\}$. 
Note that the condition 2) implies that the function $\Delta_u$ is well-defined and does not depend on the choice of points $x_2, \dots, x_n$, $y_1, \dots, y_n$. 
Furthermore, the function $\Delta_u$ satisfies the following property: if $x \sim_X x'$ and $x' \sim_X x''$, then 
$\Delta_u(x, x'') = \Delta_u(x, x') + \Delta_u(x', x'')$. 

We similarly define the equivalence relation $\sim_Y$ on the set $Y$: 
$y \sim_Y y'$ if there exist points $y_1 = y, y_2, \dots, y_{n+1} = y' \in Y$ and $x_1, \dots, x_n \in X$ such that \linebreak
$(x_i, y_i), (x_i, y_{i + 1}) \in A$ for any $i \in \{1, \dots, n\}$.  
For any $y, y' \in Y$ with $y \sim_Y y'$, we set
\begin{equation}\label{Delta_v}
\Delta_v(y, y') = \sum_{i = 1}^{n} (h(x_i, y_i) - h(x_i, y_{i+1})), 
\end{equation}
where $y_1 = y, y_2, \dots, y_{n+1} = y' \in Y$ and $x_1, \dots, x_n \in X$ are such that \linebreak
$(x_i, y_i), (x_i, y_{i + 1}) \in A$ for all $i \in \{1, \dots, n\}$. 
Then the condition 2) implies that the function $\Delta_v$ is well-defined.

Prove that for any points $(x, y), (x', y') \in A$ such that $x \sim_X x'$, $y \sim_Y y'$, we have
\begin{equation}\label{Delta_h}
h(x, y) - h(x', y') = \Delta_u(x, x') + \Delta_v(y, y'). 
\end{equation}
Indeed, there exist points $x_1 = x, x_2, \dots, x_{n+1} = x' \in X$ and $y_1, \dots, y_n \in Y$ such that 
$(x_i, y_i), (x_{i+1}, y_i) \in A$ for any $i \in \{1, \dots, n\}$ and $\Delta_u(x, x')$ is defined by the formula (\ref{Delta_u}).   
Then from (\ref{Delta_v}) it follows that 
$$
\Delta_v(y, y') = h(x, y) - \sum_{i = 1}^{n} (h(x_i, y_i) - h(x_{i+1}, y_i)) - h(x', y'). 
$$
Therefore, $\Delta_u(x, x') + \Delta_v(y, y') = h(x, y) - h(x', y')$. 

Let $X_0$ be a set of representatives of equivalence classes for the relation $\sim_X$. 
For every $x \in X$ denote by $r_X(x)$ the representative of the equivalence class containing the point $x$. 

Let $$\tilde Y = \{y \in Y: \exists x \in X_0, \, (x, y) \in A\}.$$
Note that for any $y \in Y$ there exists at most one $x \in X_0$ such that $(x, y) \in A$ (since for any $x_1, x_2 \in X$ with $(x_1, y), (x_2, y) \in A$ we have $x_1 \sim_X x_2$). 
For every $y \in \tilde Y$ denote by $x_{*}(y)$ the unique element $x \in X_0$ such that $(x, y) \in A$. 
Moreover, for any $y \in Y$ there exists $\tilde y \in \tilde Y$ such that $y \sim_Y \tilde y$. Indeed, for any $y \in Y$ there exists $x \in X$ such that $(x, y) \in A$ and afterwards we can pick
$\tilde y \in Y$ such that $(r_X(x), \tilde y) \in A$. Then we obtain that $\tilde y \in \tilde Y$ and $y \sim_Y \tilde y$. 

Let $Y_0 \subset \tilde Y$ be a set of representatives of equivalence classes for the relation~$\sim_Y$. 
For every $y \in Y$ denote by $r_Y(y)$ the representative of the equivalence class containing the point $y$.

Define
$$
u(x) = \Delta_u(x, r_X(x)), \quad x \in X, 
$$
and 
$$
v(y) =  \Delta_v(y, r_Y(y)) + h(x_{*}(r_Y(y)), r_Y(y)), \quad y \in Y. 
$$ 

Show that $u(x) + v(y) = h(x, y)$ for all points $(x, y) \in A$. 
Indeed, let $(x, y) \in A$. Then $x_{*}(r_Y(y)) = r_X(x)$, since $x \sim_X x_{*}(r_Y(y))$, $x \sim_X r_X(x)$ 
and $x_{*}(r_Y(y)), r_X(x) \in X_0$.  
By (\ref{Delta_h}) we have
$$
h(x, y) - h(x_{*}(r_Y(y)), r_Y(y)) = \Delta_u(x, x_{*}(r_Y(y))) + \Delta_v(y, r_Y(y)),
$$
hence $h(x, y) = u(x) + v(y)$. 

\end{proof}

In the case where $\eta$ is a product measure $\eta = \eta_1 \otimes \eta_2$, $\eta_1 \in \mathcal P(X)$, $\eta_2 \in \mathcal P(Y)$, we can show that in the statement of the criterion for the non-degeneracy of a cost function 
it is sufficient to test only quadruples of points 
$(x_1, y_1), (x_1, y_2), (x_2, y_1), (x_2, y_2) \in A$, where $x_1, x_2 \in X$, $y_1, y_2 \in Y$.  
Specifically, the following proposition holds.
 
\begin{proposition}\label{prop_nondegenerate_equiv2}
Let $X, Y$ be Souslin spaces, let $\eta_1 \in \mathcal P(X)$, $\eta_2 \in \mathcal P(Y)$
and $\eta = \eta_1 \otimes \eta_2$. The following conditions are equivalent 
for a Borel measurable function $h \colon X \times Y \to \mathbb R$: 

1) There exist universally measurable functions
$u \colon X \to \mathbb R$ and $v \colon Y \to \mathbb R$ such that 
$$\eta(\{(x, y) \in X \times Y: h(x, y) = u(x) + v(y)\}) > 0. $$ 

2) There exists a Borel set $A \subset X \times Y$ with $\eta(A) > 0$ such that
for any points $x_1, x_2 \in X$, $y_1, y_2 \in Y$ which satisfy the condition that
$$(x_1, y_1), (x_1, y_2), (x_2, y_1), (x_2, y_2) \in A$$ 
the following equality holds: 
$$h(x_1, y_1) + h(x_2, y_2) = h(x_1, y_2) + h(x_2, y_1).$$ 
\end{proposition}

\begin{proof}
Obviously, 1) $\Rightarrow$ 2). Let us prove that 2) $\Rightarrow$ 1). Let $A \subset X \times Y$ be a Borel set with
$\eta(A) > 0$ such that $h(x_1, y_1) + h(x_2, y_2) = h(x_1, y_2) + h(x_2, y_1)$ for any points 
$(x_1, y_1), (x_1, y_2), (x_2, y_1), (x_2, y_2) \in A$. 

Fix $\varepsilon > 0$. 
Then there exist Borel sets $X_0 \subset X$ and $Y_0 \subset Y$ 
such that $$\eta(A \cap (X_0 \times Y_0)) \ge (1 - \varepsilon^2) \eta(X_0 \times Y_0) > 0.$$ 
Let 
$$
B = A \cap (X_0 \times Y_0).
$$ 
Denote $B_x = \{y \in Y_0: (x, y) \in B\}$ for any $x \in X$. Then by Fubini's theorem
$$
\eta(B) = \int_{X_0} \eta_2(B_x) \eta_1(dx).
$$

Define
$$
X_1 = \{x \in X_0: \eta_2(B_x) \ge (1 - \varepsilon) \eta_2(Y_0)\}.
$$
We show that $\eta_1(X_1) \ge (1 - \varepsilon) \eta_1(X_0).$
Indeed, we have
\begin{multline*}
\eta((X_0 \times Y_0) \setminus B) = \int_{X_0} \eta_2(Y_0 \setminus B_x) \eta_1(dx) \ge \\ \ge
\int_{X_0 \setminus X_1} \eta_2(Y_0 \setminus B_x) \eta_1(dx) \ge \varepsilon \eta_2(Y_0) \eta_1(X_0 \setminus X_1). 
\end{multline*}
On the other hand, 
$$
\eta((X_0 \times Y_0) \setminus B) \le \varepsilon^2 \eta(X_0 \times Y_0) = \varepsilon^2 \eta_1(X_0) \eta_2(Y_0).
$$
Therefore, $\eta_1(X_0 \setminus X_1) \le \varepsilon \eta_1(X_0)$. 

Let 
$$C = B \cap (X_1 \times Y_0). 
$$
Then $\eta(C) > 0$ and 
$$\eta_2(C_x) \ge (1 - \varepsilon) \eta_2(Y_0) \quad \forall x \in X_1, $$ 
where $C_x = \{y \in Y_0: (x, y) \in C\}$. 
This implies that for any $x_1, x_2, x_3 \in X_1$ we have 
\begin{equation} \label{intersect}
\eta_2(C_{x_1} \cap C_{x_2} \cap C_{x_3}) > 0. 
\end{equation}
Indeed, otherwise we get
$$
\eta_2(C_{x_1}) + \eta_2(C_{x_2}) + \eta_2(C_{x_3}) = \int_{Y_0} (I_{C_{x_1}} + I_{C_{x_2}} + I_{C_{x_3}}) d \eta_2 
\le 2 \eta_2(Y_0), 
$$
which leads to a contradiction for small $\varepsilon > 0$, since
$$\eta_2(C_{x_1}) + \eta_2(C_{x_2}) + \eta_2(C_{x_3})  \ge 3 (1 - \varepsilon) \eta_2(Y_0).$$

Fix $x_0 \in X_1$. For every $x \in X_1$ we define
$$u(x) = h(x, y) - h(x_0, y)$$ for any $y \in Y_0$ such that $(x, y) \in C$ and $(x_0, y) \in C$. 
Such $y$ exists because $C_x \cap C_{x_0} \neq \varnothing$. Moreover, this definition is independent of the choice of $y$
due to the condition 2). 

For every $y \in Y_0$ we define 
$$
v(y) = h(x, y) - u(x)
$$
for any $x \in X_1$ such that $(x, y) \in C$. Show that this definition is correct, i.e., if $x_1, x_2 \in X_1$ are such that
$(x_1, y), (x_2, y) \in C$, then 
$$h(x_1, y) - u(x_1) = h(x_2, y) - u(x_2). $$
By virtue of (\ref{intersect}) we can take $y_0 \in C_{x_0} \cap C_{x_1} \cap  C_{x_2}$. 
Then $u(x_1) = h(x_1, y_0) - h(x_0, y_0)$ and $u(x_2) = h(x_2, y_0) - h(x_0, y_0)$. 
Therefore, 
$$
(h(x_1, y) - u(x_1)) - (h(x_2, y) - u(x_2)) = h(x_1, y) - h(x_1, y_0) - h(x_2, y) + h(x_2, y_0) = 0, 
$$
since $(x_1, y), (x_1, y_0), (x_2, y), (x_2, y_0) \in C$. 

Show that the functions $u, v$ are universally measurable. Indeed, for any $c \in \mathbb R$ the set
$\{x \in X_1: u(x) < c\}$ equals the projection of the Borel set 
$$\{(x, y) \in C: h(x, y) - h(x_0, y) < c\}$$ on $X_1$ 
and hence is Souslin and, therefore, universally measurable. Similarly, the sets $\{y \in Y_0: v(y) < c\}$ are universally measurable for all $c \in \mathbb R$. 

Moreover, by construction $h(x, y) = u(x) + v(y)$ for all $(x, y) \in C$. 
Thus, we obtain universally measurable functions $u$ and $v$, satisfying the condition 1). 
\end{proof}

Note that in the general case the non-degeneracy of a cost function $h$ 
with respect to a measure $\eta \in \mathcal P(X \times Y)$ 
is not equivalent to the fulfillment of the condition 2) from Proposition \ref{prop_nondegenerate_equiv1} for all $n \le N$, 
where $N \in \mathbb N$ is a fixed number. Consider the following example.
 
\begin{example}
{\rm
Let $X = Y = [0, 1]$ and let $N \in \mathbb N$ be fixed. We give an example of a Borel measurable function $h \colon X \times Y \to \mathbb R$ and a measure $\eta \in \mathcal P(X \times Y)$ such that
$h$ can not be expressed as a sum of functions $u(x) + v(y)$ on a set of positive $\eta$-measure, but the condition 2) from Proposition \ref{prop_nondegenerate_equiv1} is fulfilled for all $n < N$. 

Define a measure $\eta \in \mathcal P([0, 1]^2)$ by induction. 
First, consider the partition 
$$[0, 1)^2 = \bigsqcup_{i, j = 0}^{N - 1} [i/N, (i+1)/N) \times [j/N, (j + 1)/N).$$ 

Set
$$
\eta([i/N, (i+1)/N) \times [j/N, (j + 1)/N)) = \begin{cases} 
1/(2N) &\text{if $i = j$ or  $i = (j + 1) \mod N$,}\\ 
0 & \text{otherwise.}
\end{cases}
$$

Then we make a similar partition into $N^2$ small subsquares for each of the squares $[i/N, (i+1)/N) \times [j/N, (j+1)/N)$, 
where $i = j$ or  $i = (j + 1) \mod N$, 
and assign equal positive measures $1/(2N)^2$ to all of $2N$ subsquares as shown in the figure. Proceeding by induction we define the measure $\eta$. 

\begin{center}
 \includegraphics[width=0.4\textwidth]{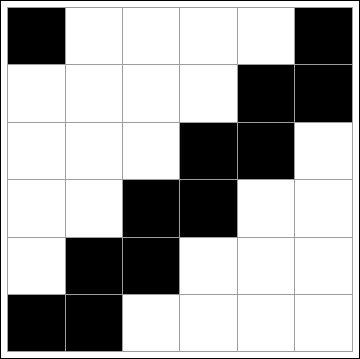}
\end{center}

Equivalently, the measure $\eta$ can be constructed in the following way. 
Let $\tau$ be a measure on $\{0, \dots, N - 1\}^2$ which is uniformly distributed on the set
$$
S = \{(0, 0), (1, 1), \dots, (N - 1, N - 1)\} \cup \{(1, 0), (2, 1), \dots, (N - 1, N - 2), (0, N - 1)\},
$$
i.e. 
$$
\tau(i, j) = 1/(2N) \quad \forall (i, j) \in S, \qquad \tau(i, j) = 0 \quad \forall (i, j) \notin S. 
$$
Let $(x^{(k)}, y^{(k)})$ be indepent random vectors with distribution $\tau$. 
Let $\eta$ be the distribution of the random vector $(x, y)$, where 
$$
x = \sum_{k  = 1}^{\infty} \frac{x^{(k)}}{N^k}, \quad y = \sum_{k  = 1}^{\infty} \frac{y^{(k)}}{N^k}. 
$$

Let 
$$
h(x, y) = \sum_{k = 1}^{\infty} I(x^{(k)} = 0, y^{(k)} = 0) 2^{-k}, 
$$
where $I(x^{(k)} = 0, y^{(k)} = 0)$ is the indicator function of the set $\{x^{(k)} = 0, y^{(k)} = 0\}$. 

Let us show that there does not exist a set $A \subset X \times Y$ with $\eta(A) > 0$ such that $h(x, y) = u(x) + v(y)$ on $A$. 
Suppose that such a set $A$ exists. 
First, we prove that $$\eta(A) \le \frac{2N - 1}{2N}. $$  
To see this, note that for any $(x, y) \in [0, 1/N)^2$ at least one of $2N$ points \linebreak
$(x + i/N, y + j/N)$, where $(i, j) \in S$, does not belong to $A$. Indeed, in the opposite case we have $(x + i/N, y + j/N) \in A$ for all $(i, j) \in S$ and
$$
\sum_{i = 0}^{N - 1} \Bigl(h(x +i/N, y + i/N) - h(x + ((i + 1) \mod N)/N, y + i/N)\Bigr) = 1, 
$$
which gives a contradiction. 

Denote
$$
A_{i, j} = A \cap [i/N, (i + 1)/N) \times [j/N, (j + 1)/N), \quad i, j \in \{0, 1, \dots, N - 1\}. 
$$ 
Then
\begin{multline*}
\eta(A) = \sum_{i, j} \eta(A_{i, j}) = \int_{[0, 1/N)^2} \sum_{(i, j) \in S} I((x + i/N, y + j/N) \in A) \eta(dx dy) \le \\
\le (2N - 1) \eta([0, 1/N)^2) = \frac{2N - 1}{2N}. 
\end{multline*}
Similarly, we obtain that for any square $Q$ with coordinates from the set $\{k/N^l: k, l \in \mathbb Z_{+}\}$ 
we have $$\eta(A \cap Q) \le \frac{2N - 1}{2N} \eta(Q).$$
This leads to a contradiction with the fact that $\eta(A) > 0$, since for any $\varepsilon > 0$ there exists a square $Q$ with coordinates from the set $\{k/N^l: k, l \in \mathbb Z_{+}\}$ 
such that 
$$\eta(A \cap Q) > (1 - \varepsilon) \eta(Q) > 0.$$
Therefore, $h$ can not be expressed as a sum of functions $u(x) + v(y)$ on a set of positive $\eta$-measure. 

We now prove that the function $h$ satisfies the condition 2) from Proposition \ref{prop_nondegenerate_equiv1}  for all $n < N$. Let 
$$A = \{(x, y) \in X \times Y: (x^{(k)}, y^{(k)}) \in S \quad \forall k \in \mathbb N\}. $$ 
Then $\eta(A) = 1$. 
Let $(x_1, y_1), \dots, (x_n, y_n)$, $(x_2, y_1), \dots, (x_n, y_{n-1}), (x_1, y_n) \in A$, where $n < N$. 
Then 
$$(x^{(k)}_i - y^{(k)}_i) \mod N \in \{0, 1\}, \quad (x^{(k)}_{i + 1} - y^{(k)}_i) \mod N \in \{0, 1\}$$ 
for all $k \in \mathbb N$ and $i \in \{1, \dots, n\}$, where 
$x_{n+1} := x_1$. 
Show that for any $k \in \mathbb N$ we have
$$
\sum_{i = 1}^n h_k(x_i, y_i) = \sum_{i = 1}^n h_k(x_{i + 1}, y_i), 
$$
where 
$$h_k(x, y) =  I(x^{(k)} = 0, y^{(k)} = 0). $$
Without limitation of generality we may assume that $x^{(k)}_1, \dots, x^{(k)}_n$ are distinct and $y^{(k)}_1, \dots, y^{(k)}_n$ are distinct. Then 
$$
(x^{(k)}_{i + 1} - x^{(k)}_i) \mod N \in \{-1, 1\} \quad \forall i \in \{1, \dots, n\}.
$$ 
Therefore, 
$$
(x^{(k)}_{i} - x^{(k)}_1) \mod N \in \{-i+1, i-1\} \quad \forall i \in \{1, \dots, n\}.
$$ 
If $n > 2$, we arrive at a contradiction, since $(x^{(k)}_n - x^{(k)}_1) \mod N \in \{-1, 1\}$ and $n < N$. 
If $n = 2$, then $(x^{(k)}_i - y^{(k)}_j) \mod N \in \{0, 1\}$ for any $i, j \in \{1, 2\}$, which also gives a contradiction. 

Therefore, the function $h$ satisfies the condition 2) for all $n < N$. 
}
\end{example}

Let us formulate a sufficient condition for the non-degeneracy of a cost function on $\mathbb R^2$. Denote by $\lambda_2$ the Lebesgue measure on $\mathbb R^2$. 

\begin{proposition}\label{prop_nondegenerate_r2}
Let $h \colon \mathbb R^2 \to \mathbb R$ be a Borel measurable function. Suppose that at least one of the following two conditions holds:
\begin{itemize}
\item[1)] For $\lambda_2$-a.e. $(x_0, y_0) \in \mathbb R^2$ 
there exists a neighbourhood $U$ of the point $(x_0, y_0)$ such that 
for all $(x, y) \in U$ there exists the mixed second partial derivative $\frac{\partial^2 h}{\partial x \partial y}(x, y) \neq 0$. 
\item[2)] There exists the Sobolev derivative $\frac{\partial^2 h}{\partial x \partial y} \in L^1_{loc}(\mathbb R^2)$ 
and $\frac{\partial^2 h}{\partial x \partial y} \neq 0$ $\lambda_2$-a.e.  
\end{itemize} 
Then for any set $A \subset \mathbb R^2$ with Lebesgue measure $\lambda_2(A) > 0$ 
there exist such points $x_1, x_2, y_1, y_2 \in \mathbb R$ that 
$$(x_1, y_1), (x_1, y_2), (x_2, y_1), (x_2, y_2) \in A$$ 
and
$$h(x_1, y_1) + h(x_2, y_2) < h(x_1, y_2) + h(x_2, y_1).$$ 
\end{proposition}

\begin{corollary}\label{cor1}
The statement of Proposition \ref{prop_nondegenerate_r2} holds true in the case where 
$h \colon \mathbb R^2 \to \mathbb R$ is an analytic function, 
which cannot be represented as a sum of functions $u(x) + v(y)$. 
\end{corollary}

\begin{proof}[Proof of Corollary \ref{cor1}]
Indeed, the function $\frac{\partial^2 h}{\partial x \partial y}$ is analytic on $\mathbb R^2$, and its set of zeros 
 has Lebesque measure 0, if $\frac{\partial^2 h}{\partial x \partial y}$ is not identically $0$. 
In the latter case the function $h$ can be represented as a sum of functions $u(x) + v(y)$. 
\end{proof}

\begin{remark}
{\rm
For any closed set $Z \subset \mathbb R^n$ there exists a function $h \in C^{\infty}(\mathbb R^n)$ such that
$h > 0$ a.e. on $\mathbb R^n \setminus Z$ and all partial derivatives of all orders of the function $h$ (including the function $h$) vanish on $Z$. Indeed, 
the open set $U = \mathbb R^n \setminus Z$ may be written as $U = \bigsqcup_{k = 1}^{\infty} U_k \sqcup N$, where 
$\{U_k\}$ is a countable set of disjoint open balls and $N$ is a set of Lebesgue measure $0$. For any open ball $U_k$ take a function $h_k \in C^{\infty}(\mathbb R^n)$ such that 
$h_k > 0$ on $U_k$, $h_k = 0$ outside $U_k$ and, moreover, the absolute values of the function $h_k$ and all its partial derivatives of order at most $k$ are estimated from above by $1/k$. 
Set $h = \sum_{k = 1}^{\infty} h_k$. This series converges uniformly on $\mathbb R^n$, since 
$$
\sup_{x \in \mathbb R^n} \Bigl|\sum_{k = m}^{\infty} h_k(x)\Bigr| \le \sup_{k \ge m} \|h_k\|_{\infty} \le 1/m \to 0, \quad m \to \infty. 
$$
Similarly, we can prove that the series of partial derivatives of any order of functions $h_k$ converges uniformly on 
$\mathbb R^n$. 
Therefore, the series for the function $h$ can be differentianted termwise and $h \in C^{\infty}(\mathbb R^n)$. 

Moreover, $h > 0$ on the set $\bigsqcup_{k = 1}^{\infty} U_k$ and all partial derivatives of all orders of the function 
$h$ vanish on $\mathbb R^n \setminus (\bigsqcup_{k = 1}^{\infty} U_k)$.
}
\end{remark}

\begin{proof}[Proof of Proposition \ref{prop_nondegenerate_r2}]
Let $A \subset \mathbb R^2$ and $\lambda_2(A) > 0$. 

In case 1) for a.e. $(x_0, y_0) \in A$ the following conditions hold:
$(x_0, y_0)$ is a density point of the set $A$ and there exists a neighbourhood of the point $(x_0, y_0)$ wherein there exists the second partial derivative $\frac{\partial^2 h}{\partial x \partial y} \neq 0$. 

In case 2) for a.e. $(x_0, y_0) \in A$ the following conditions hold:
$(x_0, y_0)$ is a density point of the set $A$, $(x_0, y_0)$ is a Lebesgue point for the function $\frac{\partial^2 h}{\partial x \partial y}$ 
and $\frac{\partial^2 h}{\partial x \partial y}(x_0, y_0)~\neq~0$. 

In each of the cases 1) and 2) fix a point $(x_0, y_0) \in A$ satisfying the listed conditions. 

For any $\delta > 0$ define
$$
A_{\delta} = \{(x, y) \in B_{\delta}(x_0, y_0): (x, y), (x + \delta, y), (x, y + \delta), (x + \delta, y + \delta) \in A \}, 
$$
that is, $A_{\delta}$ is the set of all points $(x, y)$ from the ball $B_{\delta}(x_0, y_0)$ such that the point $(x, y)$ and its shifts $(x + \delta, y)$, $(x, y + \delta)$, $(x + \delta, y + \delta)$ belong to the set $A$. 

Let us estimate $\lambda_2(A_{\delta})$. Note that
\begin{multline*}
\lambda_2(\{(x, y) \in B_{\delta}(x_0, y_0): (x + \delta, y + \delta) \notin A\}) = 
\lambda_2(B_{\delta}(x_0 + \delta, y_0 + \delta) \setminus A)  \le \\
\le \lambda_2(B_{3 \delta}(x_0, y_0) \setminus A). 
\end{multline*}
A similar estimate also holds for shifts $(x + \delta, y)$ and $(x, y + \delta)$.
Hence we obtain the bound
$$
\lambda_2(A_{\delta}) \ge \lambda_2(B_{\delta}(x_0, y_0)) - 4 \lambda_2(B_{3\delta}(x_0, y_0) \setminus A).  
$$
Since $(x_0, y_0)$ is a density point of the set $A$, 
$$
\frac{\lambda_2(B_{3\delta}(x_0, y_0) \setminus A)}{\lambda_2(B_{3\delta}(x_0, y_0))} \to 0, \quad \delta \to 0.
$$
Therefore, 
$$\frac{\lambda_2(A_{\delta})}{\lambda_2(B_{\delta}(x_0, y_0))} \to 1, \quad \delta \to 0. $$

Thus, for all sufficiently small $\delta > 0$ there exists a point $(x, y) \in B_{\delta}(x_0, y_0)$ 
such that $(x, y), (x + \delta, y), (x, y + \delta), (x + \delta, y + \delta) \in A$. 

In case 1) we have
$$
h(x + \delta, y + \delta) - h(x, y + \delta) - h(x + \delta, y) + h(x, y)  = 
\frac{\partial^2 h}{\partial x \partial y}(\xi_1, \xi_2) \cdot \delta^2
$$
for some $\xi_1 \in [x, x + \delta]$, $\xi_2 \in [y, y + \delta]$. 
Note that  
$(\xi_1, \xi_2) \in B_{3\delta}(x_0, y_0)$, since $(x, y) \in B_{\delta}(x_0, y_0)$. 
Hence for all sufficiently small $\delta$
$$
h(x + \delta, y + \delta) - h(x, y + \delta) - h(x + \delta, y + \delta) + h(x, y) \neq 0. 
$$

In case 2) we have 
$$
h(x + \delta, y + \delta) - h(x, y + \delta) - h(x + \delta, y) + h(x, y)  = 
\int_{[x, x + \delta] \times [y, y + \delta]} \frac{\partial^2 h}{\partial x \partial y}(u, v) \, du dv.  
$$
Furthermore, 
\begin{multline*}
\Bigl|\frac{1}{\delta^2} \int_{[x, x + \delta] \times [y, y + \delta]} \frac{\partial^2 h}{\partial x \partial y}(u, v) \, du dv - \frac{\partial^2 h}{\partial x \partial y}(x_0, y_0) \Bigr| \le \\
\le \frac{1}{\delta^2} \int_{B_{3\delta}(x_0, y_0)} \Bigl|\frac{\partial^2 h}{\partial x \partial y}(u, v) - \frac{\partial^2 h}{\partial x \partial y}(x_0, y_0) \Bigr| du dv \to 0, \quad \delta \to 0,
\end{multline*}
since $(x_0, y_0)$ is a Lebesgue point for the function $\frac{\partial^2 h}{\partial x \partial y}$. 
Hence
$$
h(x + \delta, y + \delta) - h(x, y + \delta) - h(x + \delta, y + \delta) + h(x, y) = \delta^2 \Bigl(\frac{\partial^2 h}{\partial x \partial y}(x_0, y_0) + o(1)\Bigr) \neq 0
$$
for all sufficiently small $\delta$.

Therefore, the points $(x, y), (x + \delta, y), (x, y + \delta), (x + \delta, y + \delta) \in A$ satisfy the condition required in the statement of Proposition \ref{prop_nondegenerate_r2}.

\end{proof}

Let us state a sufficient condition for the non-degeneracy of a cost function on locally convex spaces. 

Let $X$ be a locally convex space and $\mu \in \mathcal P(X)$. For any vector $h \in X$ consider the shift of the measure $\mu$ by $h$:
$$
\mu_h(A) = \mu(A + h) \quad \mbox{for any Borel set $A \subset X$},
$$
that is, $\mu_h$ is the image of the measure $\mu$ under the mapping $x \mapsto x - h$. 
The measure $\mu$ is called continuous along the vector $h \in X$ if 
$$ 
\lim_{t \to 0} \|\mu_{th} - \mu\| = 0, 
$$
where $\|\cdot\|$ denotes the total variation norm. 

\begin{proposition}\label{prop_nondegenerate_lcs}
Let $X, Y$ be Souslin locally convex spaces, $\eta_1 \in \mathcal P(X)$, $\eta_2 \in \mathcal P(Y)$ and
$\eta = \eta_1 \otimes \eta_2$. Let $X_C = X_C(\eta_1)$ be the space of vectors of continuity of the measure $\eta_1$, 
$Y_C = Y_C(\eta_2)$ be the space of vectors of continuity of the measure $\eta_2$. 
Let $h \colon X \times Y \to \mathbb R$ be a Borel measurable function and for $\eta$-a.e. points $(x_0, y_0)$ 
there exists an open neighbourhood $W$ of the point $(x_0, y_0)$ and  
vectors $e_1 \in X_C$, $e_2 \in Y_C$ such that for all $(x, y) \in W$ 
there exists the second order partial derivative $\frac{\partial^2 h}{\partial e_1 \partial e_2}(x, y) \neq 0$. 
Then the function $h$ satisfies the condition of non-degeneracy. 
\end{proposition}

\begin{proof}
Let us prove that for any set $A \subset X \times Y$ with $\eta(A) > 0$ there exist points $x_1, x_2 \in X$, $y_1, y_2 \in Y$ 
such that
$(x_1, y_1), (x_1, y_2), (x_2, y_1), (x_2, y_2) \in A$ and 
$$h(x_1, y_1) + h(x_2, y_2) < h(x_1, y_2) + h(x_2, y_1).$$
Let $A \subset X \times Y$ and $\eta(A) > 0$. Take a compact set $K \subset A$ with $\eta(K) > 0$. 
From the condition imposed on the function $h$ and compactness of $K$ it follows that 
there exists an open set $W \subset X \times Y$ with $\eta(K \cap W) > 0$ 
and vectors $e_1 \in X_C$, $e_2 \in Y_C$ such that for all $(x, y) \in W$ 
there exists the second partial derivative $\frac{\partial^2 h}{\partial e_1 \partial e_2}h(x, y) \neq 0$. 
Set $\tilde A = K \cap W$. 
 
Let $X = U \oplus \mathbb R e_1$, $Y = V \oplus \mathbb R e_2$. There exist conditional measures 
$\eta^{u, v}$ on $\mathbb R^2$ for the measure $\eta$ with respect to its projection $\pi$ on $U \times V$ 
such that for any Borel set $B \subset X \times Y$ we have
$$
\eta(B) = \int_{U \times V} \eta^{u, v}(B_{u, v}) \pi(du dv),
$$
where 
$$B_{u, v} = \{(s, t) \in \mathbb R^2: (u + s e_1, v + t e_2) \in B\}$$ is a two-dimensional section of the set $B$.
Since the measure $\eta$ is continuous along vectors $e_1$ and $e_2$, the conditional measures $\eta^{u, v}$ are absolutely continuous with respect to the Lebesgue measure on $\mathbb R^2$ for $\pi$-a.e. $(u, v)$ (see \cite{B10}). 
Since $\eta(\tilde A) > 0$, there exists a point $(u, v) \in U \times V$ such that
$\eta^{u, v}(\tilde A_{u, v}) > 0$ and the measure $\eta^{u, v}$ is absolutely continuous with respect to the Lebesgue measure on $\mathbb R^2$. 
Set 
$$\hat h(s, t) = h(u + s e_1, v + t e_2), \quad s, t \in \mathbb R. $$
By Proposition \ref{prop_nondegenerate_r2} there exist points $s_1, s_2, t_1, t_2 \in \mathbb R$ such that \linebreak
$(s_1, t_1), (s_1, t_2), (s_2, t_1), (s_2, t_2) \in \tilde A_{u, v}$ and 
$
\hat h(s_1, t_1) + \hat h(s_2, t_2) < \hat h(s_1, t_2) + \hat h(s_2, t_1).
$
Then the points $u + s_1 e_1$, $u + s_2 e_1$, $v + t_1 e_2$, $v + t_2 e_2$ satisfy the required condition. 
Thus the function $h$ satisfies the condition of non-degeneracy. 
\end{proof}

In case of the Euclidean space $\mathbb R^n$ we get the following sufficient condition for the non-degeneracy of a cost function. 

\begin{corollary}
Let $X = \mathbb R^n$, $Y = \mathbb R^m$. Let
measures $\eta_1 \in \mathcal P(\mathbb R^n)$ and $\eta_2 \in \mathcal P(\mathbb R^m)$ be absolutely continuous with respect to the Lebesgue measure and $\eta = \eta_1 \otimes \eta_2$. 
Let $h \colon \mathbb R^n \times \mathbb R^m \to \mathbb R$ be a Borel measurable function such that for 
$\eta$-a.e. points $(x_0, y_0)$ there exist indices $i \in \{1, \dots, n\}$ and $j \in \{1, \dots, m\}$ such that
in some neighbourhood of $(x_0, y_0)$ there exist the mixed second partial derivative 
$\frac{\partial^2 h}{\partial x_i \partial y_j}(x, y) \neq 0$. 
Then the function $h$ satisfies the condition of non-degeneracy.  
\end{corollary}

Let us give an example of a cost function which satisfies the non-degeneracy condition on an infinite-dimensional locally convex space. 

\begin{example}
{\rm
Let $X = X_0 \oplus X_1$, $Y = Y_0 \oplus Y_1$, where $X_0, Y_0$ are finite-dimensional linear subspaces. 
Let measures $\eta_1 \in \mathcal P(X)$ and $\eta_2 \in \mathcal P(Y)$ be such that
$X_0 \subset X_C(\eta_1)$, $Y_0 \subset Y_C(\eta_2)$.
Set
$$h(x_0 \oplus x_1, y_0 \oplus y_1) = h_0(x_0, y_0) + h_1(x_1, y_0 \oplus y_1) + h_2(x_0 \oplus x_1, y_1)$$ 
for all $x_0 \in X_0$, $x_1 \in X_1$, $y_0 \in Y_0$, $y_1 \in Y_1$, where the function $h_0$ is defined on the finite-dimensional space 
$X_0 \times Y_0$ and satisfies the condition that for a.e. (with respect to the Lebesgue measure) points $(x_0, y_0) \in X_0 \times Y_0$ 
there exist vectors $e_1 \in X_0$, $e_2 \in Y_0$ such that
in some neighbourhood of $(x_0, y_0)$ there exists the second partial derivative 
$\frac{\partial^2 h_0}{\partial e_1 \partial e_2}(x, y) \neq 0$. 
Then the function $h$ satisfies the condition of non-degeneracy. 
}
\end{example}

\begin{example}
{\rm
Let $X, Y$ be Souslin spaces, and let $\eta_1 \in \mathcal P(X)$, $\eta_2 \in \mathcal P(Y)$ be non-atomic measures, $\eta = \eta_1 \otimes \eta_2$. 
Let $h(x, y) = f(x) g(y)$, where $f \colon X \to \mathbb R$ and $g \colon Y \to \mathbb R$ are Borel measurable functions such that for $\eta_1$-a.e. $x \in X$ the function $f$ is injective in some neighbourhood of the point $x$ 
and for $\eta_2$-a.e. $y \in Y$ the function $g$ is injective in some neighbourhood of the point $y$. 
We show that the function $h$ satisfies the non-degeneracy condition. 

Indeed, let $A \subset X \times Y$ and $\eta(A) > 0$. Take a compact set $K \subset A$ with $\eta(K) > 0$. 
Then there exist open sets $U \subset X$ and $V \subset Y$ such that $f$ is injective on $U$, $g$ is injective on $V$ 
and $\eta(K \cap (U \times V)) > 0$. Furthermore, there exist $x_1, x_2 \in X$, $y_1, y_2 \in Y$ such that $x_1 \neq x_2$, $y_1 \neq y_2$ and 
$(x_1, y_1), (x_1, y_2), (x_2, y_1), (x_2, y_2) \in K \cap (U \times V)$. We have
$$
h(x_1, y_1) - h(x_1, y_2) - h(x_2, y_1) + h(x_2, y_2) = (f(x_1) - f(x_2))(g(y_1) - g(y_2)) \neq 0,$$
since $x_1, x_2 \in U$, $y_1, y_2 \in V$ and the functions $f$, $g$ are injective on $U$ and $V$ respectively.

Moreover, in this example we may assume that there exist Borel sets $X_i \subset X$, $i \in \mathbb N$, such that
$f$ is injective on $X_i$ for every $i \in \mathbb N$ and $\eta_1(X \setminus \bigcup_{i = 1}^{\infty} X_i) = 0$, 
and there exist Borel sets $Y_i \subset Y$, $i \in \mathbb N$, such that
$g$ is injective on $Y_i$ for every $i \in \mathbb N$ and $\eta_2(Y \setminus \bigcup_{i = 1}^{\infty} Y_i) = 0$. 
Then for any $A \subset X \times Y$ with $\eta(A) > 0$ there exist $i, j \in \mathbb N$ such that 
$\eta(A \cap (X_i \times Y_j)) > 0$. Therefore, we can find points 
$(x_1, y_1), (x_1, y_2), (x_2, y_1), (x_2, y_2) \in A \cap (X_i \times Y_j)$ with
$$
h(x_1, y_1) + h(x_2, y_2) \neq h(x_1, y_2) + h(x_2, y_1). 
$$
}
\end{example}

\end{document}